\documentclass[11pt, reqno]{amsart}
\usepackage{amsmath,amsthm,amssymb,mathrsfs}
\usepackage[abbrev,non-sorted-cites]{amsrefs}
\usepackage{color}
\usepackage{verbatim}
\usepackage{datetime}
\usepackage{hyperref}

\theoremstyle{plain}
  \newtheorem{thm}{Theorem}
  \newtheorem{lem}[thm]{Lemma}

\theoremstyle{definition}
  \newtheorem{defn}[thm]{Definition}
	
  \newtheorem{rmk}[thm]{Remark}
  \newtheorem*{ack*}{Acknowledgement}
  \newtheorem*{ques*}{Question}
\theoremstyle{plain}

\numberwithin{equation}{section}
\newcommand\ip[2]{\langle{#1},{#2}\rangle}

\newcommand\BR{\mathbb{R}}

\newcommand\RO{\mathrm{O}}
\newcommand\II{\mathrm{II}}

\newcommand\be{\mathbf{e}}
\newcommand\bff{\mathbf{f}}
\newcommand\bfI{\mathbf{I}}

\newcommand\Sm{\Sigma}
\newcommand\Dt{\Delta}
\newcommand\Om{\Omega}
\newcommand\af{\alpha}
\newcommand\ld{\lambda}
\newcommand\dt{\delta}
\newcommand\sm{\sigma}
\newcommand\vep{\varepsilon}

\newcommand\pl{\partial}

\newcommand\dd{\mathrm{d}}
\newcommand\dv{\mathrm{dvol}}

\newcommand\bx{\bar{x}}
\newcommand\by{\bar{y}}

\DeclareMathOperator{\vol}{Vol}
\DeclareMathOperator{\tr}{tr}
\DeclareMathOperator{\spn}{span}

\begin{document}

\title{An Isoperimetric-Type Inequality for Spacelike Submanifold in the Minkowski Space}

\author{Chung-Jun Tsai}
\address{Department of Mathematics\\National Taiwan University\\Taipei 10617\\ Taiwan}
\email{cjtsai@ntu.edu.tw}

\author{Kai-Hsiang Wang}
\address{National Center for Theoretical Sciences, Mathematics Division\\National Taiwan University\\Taipei 10617\\ Taiwan}
\email{akoujp@gmail.com}

\thanks{This research is supported in part by Taiwan Ministry of Science and Technology grants 108-2636-M-002-009, 109-2636-M-002-007 and National Center for Theoretical Science.}

\begin{abstract}
We prove an isoperimetric-type inequality for maximal, spacelike submanifold in the Minkowski space.  The argument is based on the recent work of Brendle.
\end{abstract}

\maketitle


\section{Introduction}

The isoperimetric inequality asserts that there exists a constant $c_n$ such that
\begin{align}
\vol(\Sm)^{n-1} \leq c_n \vol(\pl\Sm)^n
\label{E_isop} \end{align}
for any compact, $n$-dimensional, minimal submanifold $\Sm$ in the Euclidean space, $\BR^{n+m}$.  The sharp constant of $c_n$ is believed to be
\begin{align}
\frac{\vol(B^n)^{n-1}}{\vol(\pl B^n)^n} = \frac{1}{n^n\vol(B^n)}
\label{E_const} \end{align}
where $B^n$ is the unit ball in $\BR^n$.  The sharp inequality is known to be true in the full dimensional case ($m=0$), and the surface case ($n=2$) with some topological condition.  We refer to \cite[Problem 109]{ref_Yau} and \cite{ref_Ch} for the developments.  Recently, Brendle \cite{ref_Brendle} generalizes brilliantly the argument of \cite{ref_Ca}, and proves the sharp inequality for codimension at most $2$ ($m\leq2$).

The purpose of this note is to investigate the isoperimetric inequality for spacelike submanifolds in the Minkowski space based on the approach of Brendle.  What follows are the main results.

\begin{thm} \label{mthm1}
For any $k\geq 2$ and $\tau\in[1,\infty)$, suppose that $\Sm$ is a compact, connected, $n$-dimensional, spacelike submanifold in $\BR^{n,k}$ whose space-time slope is no greater than $\tau$.  Then, it satisfies
\begin{align*}
\vol(\Sm)^{n-1} &\leq c(n,k,\tau) \left( \vol(\pl\Sm) + \int_\Sm\sqrt{-|H|^2}\dv \right)^{n} ~,
\end{align*}
where $H$ is the (timelike) mean curvature vector of $\Sm$, and
\begin{align*}
c(n,k,\tau) &= \frac{n+k-2}{n}(\tau + \sqrt{\tau^2-1})^{n+k-2}\frac{1}{n^n \vol(B^n)} ~.
\end{align*}
\end{thm}

Here are some remarks.
\begin{enumerate}
\item The quantity $\tau\geq 1$ measures the deviation of $\Sm$ from belonging to an Euclidean slice.  If there is some tangent direction close to being null, $\tau$ is large.  See Definition \ref{def_nonflat}.
\item For $k=1$, one can embed $\BR^{n,1}$ into $\BR^{n,2}$, and Theorem \ref{mthm1} applies.
\item Suppose that $\Sm$ is an $n$-dimensional domain in $\BR^{n}$.  In this case, $\tau$ can be taken to be $1$.  By embedding $\BR^{n}$ into $\BR^{n,2}$, Theorem \ref{mthm1} recovers the sharp form of \eqref{E_isop} (the full dimension case).
\end{enumerate}

Unlike the Euclidean case, our constant has an additional dependence on $\tau$.  This issue will be addressed in Section \ref{sec_const}.  The observation there illustrates the difficulty to have the dependence only on the dimensions.

\begin{thm} \label{mthm2}
For any $k\geq 2$, $m\geq1$ and $\tau\in[1,\infty)$, suppose that $\Sm$ is a compact, connected, $n$-dimensional, spacelike submanifold in $\BR^{n+m,k}$ whose space-time slope is no greater than $\tau$.  Then, it obeys
\begin{align*}
\vol(\Sm)^{n-1} &\leq c(n,m,k,\tau) \left( \vol(\pl\Sm) + \int_\Sm \left( \frac{1+\sqrt{\tau^4-1}}{\tau}|\pi_s(H)| + \sqrt{\tau^2+1}|\pi_t(H)| \right)\dv \right)^{n} ~.
\end{align*}
Here, $\pi_s$ is the orthogonal projection onto $\BR^{n+m,0}$, $\pi_t$ is the orthogonal projection onto $\BR^{0,k}$, and
\begin{align*}
c(n,m,k,\tau) &= \frac{n+m+k-2}{n+m} \frac{(\tau + \sqrt{\tau^2-1})^{n+m+k-2}(\tau^2+1)^{\frac{k-2}{2}}}{\tau^{m+k-2}}\frac{\vol(B^m)}{n^n\vol(B^{n+m})} ~.
\end{align*}
The convention here is that $|\pi_t(H)|\geq0$.
\end{thm}

\begin{ack*}
Both authors would like to thank Mu-Tao Wang for helpful discussions.
\end{ack*}

\section{Some Linear Algebra}

For the Minkowski space $\BR^{n+m,k}$, denote by $\pi_s$ the projection onto the $\BR^{n+m,0}$-summand, and by $\pi_t$ the projection onto the $\BR^{0,k}$-summand.  For vectors after $\pi_t$, the convention here is to take the \emph{positive definite} inner product.  That is to say, $|v|^2 = |\pi_s(v)|^2 - |\pi_t(v)|^2$.

\begin{defn} \label{def_nonflat}
For a compact, spacelike submanifold $\Sm$ of $\BR^{n+m,k}$, let
\begin{align*}
\tau(x) = \max\{\, |\pi_s(v)| \,:\, v\in T_x\Sm ~, |v| = 1 \,\} \geq 1 ~,
\end{align*}
and let $\tau$ be $\max\{\, \tau(x) \,:\, x\in\Sm \,\}$.
\end{defn}
Since $|\pi_t(v)|^2 = |\pi_s(v)|^2 - |v|^2$, $\max\{\, |\pi_t(v)| \,:\, v\in T_x\Sm ~, |v| = 1 \,\}$ is equal to $\sqrt{(\tau(x))^2 - 1}$.
\begin{rmk}
This notion depends on the choice of the Euclidean slice, $\BR^{n+m,0}$.  In other words, $\tau$ is \emph{not} invariant under the isometries of $\BR^{n+m,k}$.  The infimum of $\tau$ over $\{ A(\Sm) : A\in\RO(n+m,k) \}$ measures the deviation of $\Sm$ from belonging to an Euclidean slice.
\end{rmk}

\begin{lem} \label{lem_linear}
Let $L\subset\BR^{n+m,k}$ be an $n$-dimensional spacelike subspace.  Denote $\max\{\, |\pi_s(v)| \,:\, v\in L ~, |v| = 1 \,\}$ by $\tau(L)$. Then,
\begin{enumerate}
\item $|\pi_L(v)| \leq \tau(L)|\pi_s(v)|+\sqrt{\tau(L)^2-1}|\pi_t(v)|$ for any $v \in \BR^{n+m,k}$;
\item there exists an orthogonal decomposition of $L^\perp$, $N^+\oplus N^-$ such that
\begin{itemize}
\item $N^+$ is spacelike, and $N^-$ is timelike;
\item $|\pi_{N^+}(v)| \leq |\pi_s(v)|$ for any $v \in \BR^{n+m,k}$;
\item $|\pi_{N^-}(v)| \leq \sqrt{\tau(L)^2-1}|\pi_s(v)|+\tau(L)|\pi_t(v)|$ for any $v \in \BR^{n+m,k}$.
\end{itemize}
As mentioned above, the norm on $N^-$ is taken to be the positive definite one.
\end{enumerate}
\end{lem}
\begin{proof}
Since $L$ is spacelike, $\pi_s|_L$ is injective.  It follows that $H = \pi_s(L)$ is an $n$-dimensional subspace in $\BR^{n+m,0}\subset\BR^{n+m,k}$.  Let $\{\be^+_j\}_{j=n+1}^{n+m}$ be an orthonormal basis for the orthogonal complement of $H$ in $\BR^{n+m,0}$

The subspace $L$ is a graph of a transform, $B$, from $H$ to $\BR^{0,k}\subset\BR^{n+m,k}$.  By using the positive definite metric on $\BR^{0,k}$, and applying the singular value decomposition, there exist orthonormal basis $\{\be^+_j\}_{j=1}^n$ for $H$, orthonormal basis $\{\be^-_{j}\}_{j=1}^k$ for $\BR^{0,k}$, singular values $\{\ld_j\}_{j=1}^n$ such that
\begin{align*}
\{\frac{1}{\sqrt{1-\ld_j^2}}(\be^+_j + \ld_j\be^-_j)\}_{j=1}^n  ~\text{ forms an orthogonal basis for $L$.}
\end{align*}
Note that $\ld_j^2<1$ follows from the spacelike condition of $L$.  When $n>k$, $\ld_j = 0$ for $j>k$.  By re-labeling, assume $|\ld_1|\geq|\ld_j|$ for any $j$.  It is easy to see that $\tau(L) = {1}/{\sqrt{1-\ld_1^2}}$.

For any $v\in\BR^{n+m,k}$, write $v = v^+_\mu\be^+_\mu + v^-_\nu\be^-_\nu$.  Its orthogonal projection onto $L$ is
\begin{align*}
\pi_L(v) &= \frac{1}{{1-\ld_j^2}}(v_j^+ - \ld_j v_j^-)(\be^+_j + \ld_j\be^-_j) ~,
\end{align*}
and thus
\begin{align} \begin{split}
|\pi_L(v)|^2 &= \sum_j\frac{1}{1-\ld_j^2}(v_j^+ - \ld_j v_j^-)^2 \\
&=\sum_j\frac{1}{1-\ld_j^2}\left(({v_j^+})^2+\ld_j^2(v_j^-)^2-2\ld_j v_j^+v_j^-\right)\\
&\leq \frac{1}{1-\ld_1^2}\left(\sum_j(v_j^+)^2\right)+\frac{\ld_1^2}{1-\ld_1^2}\left(\sum_j (v_j^-)^2\right)+2\frac{\ld_1}{1-\ld_1^2}\sqrt{\sum_j (v_j^+)^2}\sqrt{\sum_j (v_j^-)^2} \\
&\leq \tau(L)^2|\pi_s(v)|^2+\left(\left(\tau(L)\right)^2-1\right) \left( |\pi_t(v)|^2\right)+2\tau(L)\sqrt{\tau(L)^2-1}|\pi_s(v)||\pi_t(v)| ~.
\end{split} \label{est1} \end{align}
It proves assertion (1) of this Lemma.

For the second part, let
\begin{align*}
N^+ = \spn\{\be^+_j\}_{j=n+1}^{n+m}  \quad\text{and}\quad  N^- = \spn\{\frac{1}{\sqrt{1-\ld_j^2}}(\ld_j\be^+_j + \be^-_j)\}_{j=1}^k ~.
\end{align*}
Assertion (2) can be obtained by a similar argument as \eqref{est1}.
\end{proof}

\section{The Comparison Map}

The main ingredient of the argument in \cite{ref_Brendle} is a map from the normal bundle of $\Sm$ to the unit ball in $\BR^{n+m}$.  This section is devoted to study analogous properties of the map.

Throughout the rest of this paper, $\Sm$ always denotes a compact, connected, $n$-dimensional, spacelike submanifold in $\BR^{n+m,k}$, with boundary $\pl\Sm$.

For a positive constant $c_0$ and a function $f$ on $\Sm$, let $u$ be the solution to the following Neumann problem:
\begin{align}
\begin{cases}
\Dt^{\Sm} u = c_0(c_f - f)  &\text{on }\Sm ~, \\
\ip{\nabla^\Sm u}{\eta} = c_0  &\text{on }\pl\Sm \\
\end{cases}
\label{Neumann1} \end{align}
where $\eta$ is the unit outer normal of $\pl\Sm$ with respect to $\Sm$, and
\begin{align}
c_f &= \frac{\vol(\pl\Sm) + \int_\Sm f\,\dv}{\vol(\Sm)} ~.
\label{const1} \end{align}
The value of $c_f$ is determined by the Green's identity.  We will use Lipschitz continuous $f$, and thus $u$ will be of class $C^{2,\af}$ for any $\af\in(0,1)$ (see \cite[section 6.7]{ref_GT}).  With this function $u$, introduce the map $\Phi = \Phi_{\Sm,c_0,f}$ as follows:
\begin{align} \begin{array}{cccl}
\Phi: &N\Sm &\to &\BR^{n+m,k} \\
&(x,y) &\mapsto &\nabla^\Sm u(x) + y
\end{array} \label{map1} \end{align}
where $x\in\Sm$ and $y\in N_x\Sm = (T_x\Sm)^\perp$.

With $\tau$ given by Definition \ref{def_nonflat}, consider the following regions:
\begin{align} \begin{split}
D &= \{\, \xi \in \BR^{n+m,k} \,:\, c_0 -  \tau|\pi_s(\xi)| -\sqrt{\tau^2-1} |\pi_t(\xi)|>0 \,\} ~, \\
U &= \{\, x\in\Sm\setminus\pl\Sigma \,:\, |\nabla^\Sm u(x)| < c_0 \,\} ~, \\
\Om &= \{\, (x,y) \,:\, x \in U,\, y\in N_x\Sm,\, \nabla^\Sm u(x)+y \in D \,\} ~, \\
A &= \{\, (x,y)\in \Omega \,:\, \nabla^2_\Sm u(x) - \ip{\II|_x}{y}\geq 0 \,\}
\end{split} \label{region1} \end{align}

\begin{lem} \label{lem_map1}
The image of $A$ under the map $\Phi$ \eqref{map1} is exactly $D$.
\end{lem}
\begin{proof}
It follows from definition that $\Phi(A)\subset D$.  For surjectivity, define $w_\xi(x) = u(x) - \ip{x}{\xi}$ for any $\xi\in D$.
When $x\in\pl\Sm$, it follows from Definition \ref{def_nonflat} and Cauchy--Schwarz inequality that
\begin{align*}
\ip{\nabla^\Sm w_\xi(x)}{\eta(x)} &= \ip{\nabla^\Sm u(x)}{\eta(x)} - \ip{\xi}{\eta(x)} \\
&= c_0 - \ip{\pi_s(\xi)}{\pi_s(\eta(x))} + \ip{\pi_t(\xi)}{\pi_t(\eta(x))} \\
&\geq c_0 - \tau\,|\pi_s(\xi)| - \sqrt{\tau^2-1}\,|\pi_t(\xi)| > 0 ~.
\end{align*}

Therefore, there is an $x_\xi\in\Sm\setminus\pl\Sm$  attaining the minimum of $w_\xi$.  Since $\nabla^\Sm w_\xi(x_\xi) = 0$, $\nabla^\Sm u(x_\xi) = \xi^T$ where $\xi^T$ is the orthogonal projection of $\xi$ onto $T_x\Sm$.  Equivalently, $\xi = \nabla^\Sm u(x_\xi) + y_\xi$ for some $y_\xi\in N_x\Sm$.  According to (1) of Lemma \ref{lem_linear},
\begin{align*}
|\nabla^\Sm u(x_\xi)| &\leq \tau(x_\xi)|\pi_s(\xi)|+\sqrt{\tau(x_\xi)^2-1}|\pi_t(\xi)|\leq \tau|\pi_s(\xi)|+\sqrt{\tau^2-1}|\pi_t(\xi)| < c_0 ~,
\end{align*}
and hence $x_\xi\in U$.  This shows that $\Phi(\Om) = D$.

Since $x_\xi$ achieves the minimum of $w_\xi$,
\begin{align*}
0 &\leq \nabla^2_\Sm w_\xi(x_\xi) = \nabla^2_\Sm u(x_\xi) - \ip{\II|_{x_\xi}}{\xi} ~.
\end{align*}
Note that $\ip{\II|_{x_\xi}}{\xi} = \ip{\II_{x_\xi}}{\nabla^\Sm u(x_\xi) + y_\xi} = \ip{\II|_{x_\xi}}{y_\xi}$, and thus $(x_\xi,y_\xi)\in A$.
\end{proof}

\subsection{The Jacobian Determinant}

To extract an isoperimetric-type inequality from Lemma \ref{lem_map1}, some estimate on the Jacobian determinant of $\Phi$ is needed.  The properties in this subsection are essentially the same as \cite[Lemma 5 and 6]{ref_Brendle}.  The arguments are included for completeness.

\begin{lem} \label{lem_Jac1}
At any $(x,y)\in\Om$, $(\det \dd\Phi)(x,y) = \det(\nabla^2_\Sm u(x) - \ip{\II|_x}{y})$.
\end{lem}
\begin{proof}
Fix a point $(\bx,\by)\in\Om$.  Let $(x_1,\ldots,x_n)$ be a normal coordinate system for $\Sm$ near $\bx$, and let $\{\be_i\}_{i=1}^n$ be a local orthonormal frame for $T\Sm$ near $\bx$ such that $\be_i = \frac{\pl}{\pl x_i}$ at $\bx$ for $i\in\{1,\ldots,n\}$.  Also choose a local orthonormal frame $\{\bff^+_i\}_{i=1}^{m}\cup\{\bff^-_j\}_{j=1}^{k}$ of $N\Sm$ near $\bx$, where $\bff^+_i$'s are spacelike, and $\bff^-_j$'s are timelike.

By writing a normal vector to $\Sm$ as $y = \sum_{i=1}^m y_i\,\bff^+_i + \sum_{j=1}^k y_{m+j}\,\bff^-_j$,
$$(x_1, \ldots, x_n,y_1, \ldots, y_{m}, y_{m+1}, \ldots, y_{m+k})$$
constitutes a local coordinate system on the total space of $N\Sm$.  We compute the differential of $\Phi$ at $(\bx,\by)$:
\begin{align*}
\ip{\frac{\pl\Phi}{\pl x_i}(\bx,\by)}{\be_{j}} &= (\nabla^2_\Sigma u)(\be_i, \be_j) - \ip{\II|_x(\be_i,\be_j)}{\by} ~,
&\ip{\frac{\pl\Phi}{\pl y_i}(\bx,\by)}{\bff^+_j} &= \ip{\bff^+_i}{\bff^+_j} = \dt_{ij} ~, \\
\ip{\frac{\pl\Phi}{\pl y_i}(\bx,\by)}{\be_{j}} &= \ip{\bff^+_i}{\be_j} = 0 ~,
&\ip{\frac{\pl\Phi}{\pl y_{m+i}}(\bx,\by)}{\bff^+_j} &= \ip{\bff^-_i}{\bff^+_j} = 0 ~, \\
\ip{\frac{\pl\Phi}{\pl y_{m+i}}(\bx,\by)}{\be_j} &= \ip{\bff^-_i}{\be_j} = 0 ~,
&\ip{\frac{\pl\Phi}{\pl y_{m+i}}(\bx,\by)}{\bff^-_j} &= \ip{\bff^-_i}{\bff^-_j} = -\dt_{ij} ~.
\end{align*}
It follows that the Jacobian of $\Phi$ at $(\bx,\by)$ in terms of the frame $\{\be_1,\ldots,\be_n,\bff^+_1,\ldots\,\bff^+_m,\bff^-_1,\ldots,\bff^-_k\}$ takes the following form
\begin{align*}
\begin{bmatrix}
\nabla^2_\Sm u(\bx) - \ip{\II|_{\bx}}{\by} & 0 & 0 \\
* & \bfI_m & 0\\
* & * & \bfI_k
\end{bmatrix} ~.
\end{align*}
Taking determinant finishes the proof of this Lemma.
\end{proof}

\begin{lem} \label{lem_Jac2}
At any $(x,y)\in A$, $0 \leq (\det\dd\Phi)(x,y) \leq \left( \frac{1}{n}(c_0 (c_f - f(x)) - \ip{H(x)}{y}) \right)^n$.  In particular, $c_0(c_f - f(x)) - \ip{H(x)}{y} \geq 0$ at any $(x,y)\in A$.
\end{lem}
\begin{proof}
For $(x,y)\in A$, $\nabla^2_\Sm u(x) - \ip{\II|_x}{y}$ is positive semi-definite.  By the arithmetic-geometric inequality, 
\begin{align*}
0 \leq \det(\nabla^2_\Sm u(x) - \ip{\II|_x}{y}) \leq \left( \frac{\tr(\nabla^2_\Sm u(x) - \ip{\II|_x}{y})}{n} \right)^n = \left( \frac{\Dt^\Sm u(x) - \ip{H(x)}{y}}{n} \right)^n\ ~.
\end{align*}
Combining this with Lemma \ref{lem_Jac1} and \eqref{Neumann1} finishes the proof of this Lemma.
\end{proof}

\section{Proof of Main Theorems}

This section is devoted to the proof of Theorem \ref{mthm1} and \ref{mthm2}.  It will be done by analyzing the total (usual) Lebesgue measure of $D$ via $\Phi$.  In this section, $k$ is always assumed to be greater than or equal to $2$.

\subsection{Measure Estimate on the Target}

It follows from Lemma \ref{lem_map1} that
\begin{align}
\int_{\xi \in D} \rho(\xi) \,\dd\xi
&\leq \int_{x\in U}\!\int_{\substack{ y\in N_x\Sm \\ y+ \nabla^\Sm u(x)\in D}} (\det \dd\Phi)(x,y)\cdot 1_{A}(x,y)\cdot \rho(y+\nabla^\Sm u(x)) \,\dd y\, \dv_x ~, \label{ineq1}
\end{align}
for any non-negative $\rho\in L^1(D)$, where $1_A$ is the characteristic function of $A$.
The strategy is to choose $\rho$ to be the characteristic function of the set
\begin{align}
S_\vep = \{\xi\in\BR^{n+m,k}: -\vep^2 < |\xi|^2 = |\pi_s(\xi)|^2 - |\pi_t(\xi)|^2 < 0\} ~,
\label{region2} \end{align}
and examine the limiting inequality of \eqref{ineq1} over $\vep^2$ as $\vep\to0$.  For $\vep<\!<1$, it follows from
\begin{align*}
S_\vep\cap D \supset \{\xi: |\pi_s(\xi)|^2 < |\pi_t(\xi)|^2 < |\pi_s(\xi)|^2 + \vep^2 ~,~ |\pi_s(\xi)| < c_0\tau -\sqrt{(\tau^2-1)(c_0^2+\vep^2)} \} 
\end{align*}
that
\begin{align*}
\int_{\xi \in D} 1_{S_\vep}(\xi)\,\dd\xi \geq \vol(B^k)\vol(\pl B^{n+m}) \int_0^{c_0\tau-\sqrt{(\tau^2-1)(c_0^2+\epsilon^2)}} \left((s^2+\epsilon^2)^\frac{k}{2}-s^k\right) s^{n+m-1}\,\dd s
\end{align*}
where $s = |\pi_s(\xi)|$.  By the L'H\^{o}pital's rule,
\begin{align*}
&\quad \lim_{\vep\to0}\frac{1}{\vep^2}\int_0^{c_0\tau-\sqrt{(\tau^2-1)(c_0^2+\epsilon^2)}} \left((s^2+\epsilon^2)^\frac{k}{2}-s^k\right) s^{n+m-1}\,\dd s \\
&= \int_0^{c_0(\tau-\sqrt{\tau^2-1})} \frac{k}{2}s^{k-2}\, s^{n+m-1}\,\dd s = \frac{k}{2(n+m+k-2)} \left(c_0(\tau - \sqrt{\tau^2-1})\right)^{n+m+k-2} ~.
\end{align*}

Putting these together gives
\begin{align} \begin{split}
&\quad \frac{k\,\vol(B^k)\vol(\pl B^{n+m})}{2(n+m+k-2)} \left(c_0(\tau - \sqrt{\tau^2-1})\right)^{n+m+k-2} \\
&\leq \limsup_{\vep\to0}\frac{1}{\vep^2}\int_{x\in U}\!\int_{\substack{ y\in N_x\Sm ~,~ y+ \nabla^\Sm u(x)\in D ~, \\ -\vep^2-|\nabla^\Sm u(x)|^2 < |y|^2 < - |\nabla^\Sm u(x)|^2 }} (\det \dd\Phi)(x,y)\cdot 1_{A}(x,y)\,\dd y\, \dv_x ~.
\end{split} \label{ineq2} \end{align}

\subsection{Proof of Theorem \ref{mthm1}}

Now, suppose that $m = 0$.  Let $f$ be the length of the mean curvature vector of $\Sm$,
\begin{align}
f(x) = \sqrt{-|H|^2} ~.
\end{align}
Since $m=0$, $N\Sm$ consists of only timelike directions, and hence $|\ip{H(x)}{y}|\leq\sqrt{-|H(x)|^2}\sqrt{-|y|^2}$ for any $y\in N_x\Sm$.  According to Lemma \ref{lem_Jac2},
\begin{align*}
(\det \dd\Phi)(x,y)\cdot 1_{A}(x,y) &\leq \left( \frac{c_0(c_f - \sqrt{-|H|^2}) - \ip{H(x)}{y}}{n} \right)^n\cdot 1_{A}(x,y) \\
&\leq n^{-n} { \left(\left( c_0(c_f - \sqrt{-|H|^2}) + \sqrt{-|H|^2}\sqrt{-|y|^2} \right)_+\right)^n }
\end{align*}
where $(\;\cdot\;)_+$ means $\max\{(\;\cdot\;),0\}$.
It follows that the inner integral of \eqref{ineq2} is no greater than
\begin{align*}
&\quad n^{-n}\int_{\substack{ y\in N_x\Sm \\ -\vep^2-|\nabla^\Sm u(x)|^2 < |y|^2 < - |\nabla^\Sm u(x)|^2 }} \left(\left( c_0(c_f - \sqrt{-|H|^2}) + \sqrt{-|H|^2}\sqrt{-|y|^2} \right)_+\right)^n \,\dd y \\
&= \frac{1}{2} n^{-n}\vol(\pl B^k)\int_{|\nabla^\Sm u(x)|^2}^{|\nabla^\Sm u(x)|^2+\vep^2} \left(\left(c_0(c_f - \sqrt{-|H|^2}) + \sqrt{-|H|^2}\,t^{\frac{1}{2}} \right)_+\right)^n \,t^{\frac{k-2}{2}}\,\dd t
\end{align*}
where the parameter $t$ is $-|y|^2$.  Since $k\geq2$, the integrand is continuous for $t\geq0$.  It is clear that the limit of the above expression over $\vep^2$ as $\vep\to0$ is
\begin{align*}
\frac{1}{2} n^{-n}\vol(\pl B^k) \left(\left( c_0(c_f - \sqrt{-|H|^2}) + \sqrt{-|H|^2}\,|\nabla^\Sm u(x)| \right)_+\right)^n \,|\nabla^\Sm u(x)|^{k-2} ~.
\end{align*}

By combining these with \eqref{ineq2} and \eqref{region1},
\begin{align*}
&\quad \frac{k\,\vol(B^k)\vol(\pl B^{n})}{2(n+k-2)} \left(c_0(\tau - \sqrt{\tau^2-1})\right)^{n+k-2} \\
&\leq \frac{1}{2} n^{-n}\vol(\pl B^k) \int_{x\in U} \left(\left( c_0(c_f - \sqrt{-|H|^2}) + \sqrt{-|H|^2}\,|\nabla^\Sm u(x)| \right)_+\right)^n \,|\nabla^\Sm u(x)|^{k-2} \,\dv_x ~. \\
&\leq \frac{1}{2} n^{-n}\vol(\pl B^k)\,(c_0)^{n+k-2}\,(c_f)^{n} \vol(U)
\end{align*}
This together with \eqref{const1} and the fact that $\vol(U)\leq\vol(\Sm)$ finishes the proof of Theorem \ref{mthm1}.

\subsection{Proof of Theorem \ref{mthm2}}

When $m\geq1$, let
\begin{align}
f(x) &= \left(\frac{1}{\tau}+\sqrt{\tau^2-\frac{1}{\tau^2}}\right)|\pi_s(H)| + \sqrt{\tau^2+1}\,|\pi_t(H)| ~.
\end{align}
By Lemma \ref{lem_Jac2} and Lemma \ref{lem_linear},
\begin{align*}
&\quad (\det \dd\Phi)(x,y)\cdot 1_{A}(x,y) \\
&\leq \left( \frac{c_0(c_f - f) - \ip{H(x)}{y}}{n} \right)^n\cdot 1_{A}(x,y) \\
&\leq n^{-n} { \left(\left( c_0(c_f - f) + |\pi_{N^+_x}(H)|\,|\pi_{N^+_x}(y)| + |\pi_{N^-_x}(H)|\,|\pi_{N^-_x}(y)| \right)_+\right)^n } \\
&\leq n^{-n} { \left(\left( c_0(c_f - f) + |\pi_{s}(H)|\,|\pi_{N^+_x}(y)| + (\sqrt{\tau^2-1}|\pi_{s}(H)|+\tau|\pi_t(H)|)|\pi_{N^-_x}(y)| \right)_+\right)^n } ~.
\end{align*}
For brevity, denote $|\pi_{s}(H)|$ by $h_1(x)$ and $\sqrt{\tau^2-1}|\pi_{s}(H)|+\tau|\pi_t(H)|$ by $h_2(x)$.  In this case, $|y|^2 = |\pi_{N^+_x}(y)|^2 - |\pi_{N^-_x}(y)|^2$.  The region of the inner integral of \eqref{ineq2} is contained in
\begin{align*}
 \left\{ y\in N_x\Sm : |\pi_{N^+_x}(y)| < \frac{c_0}{\tau} , |\pi_{N^+_x}(y)|^2+|\nabla^\Sm u(x)|^2 < |\pi_{N^-_x}(y)|^2 < |\pi_{N^+_x}(y)|^2+|\nabla^\Sm u(x)|^2+\vep^2 \right\} ~.
\end{align*}
It follows that the inner integral of \eqref{ineq2} is no greater than
\begin{align*}
\frac{\vol(\pl B^m)\vol(\pl B^k)}{2\,n^n}\int_0^{\frac{c_0}{\tau}}\int_{s^2+|\nabla^\Sm u(x)|^2}^{s^2+|\nabla^\Sm u(x)|^2+\vep^2}
\left(\left( c_0(c_f - f) + h_1\,s + h_2\,t^{\frac{1}{2}} \right)_+\right)^n s^{m-1}t^{\frac{k-2}{2}}\dd t\,\dd s
\end{align*}
where $s = |\pi_{N^+_x}(y)|$ and $t = |\pi_{N^-_x}(y)|^2$.

Since $k\geq2$, the limit of the above integral over $\vep^2$ as $\vep\to0$ is
\begin{align*}
&\quad \int_0^{\frac{c_0}{\tau}} \left(\left( c_0(c_f - f) + h_1\,s + h_2\sqrt{s^2+|\nabla^\Sm u(x)|^2} \right)_+\right)^n (s^2+|\nabla^\Sm u(x)|^2)^{\frac{k-2}{2}} s^{m-1} \dd s \\
&\leq \left(\left( c_0(c_f - f) + h_1\frac{c_0}{\tau} + h_2\sqrt{(\frac{c_0}{\tau})^{2}+|\nabla^\Sm u(x)|^2} \right)_+\right)^n \left( (\frac{c_0}{\tau})^{2}+|\nabla^\Sm u(x)|^2 \right)^{\frac{k-2}{2}}\cdot \frac{1}{m}(\frac{c_0}{\tau})^m ~.
\end{align*}
At any $x\in U$, it follows from $|\nabla^\Sm u(x)| < c_0$ that
\begin{align*}
(\frac{c_0}{\tau})^{2}+|\nabla^\Sm u(x)|^2 &\leq (\frac{c_0}{\tau})^2 (\tau^2+1)\quad\text{and} \\
c_0(c_f - f) + h_1\frac{c_0}{\tau} + h_2\sqrt{(\frac{c_0}{\tau})^{2}+|\nabla^\Sm u(x)|^2} &\leq c_0(c_f - f + f) = c_0\,c_f ~.
\end{align*}

With these estimates, \eqref{ineq2} leads to
\begin{align*}
&\quad \frac{k\,\vol(B^k)\vol(\pl B^{n+m})}{2(n+m+k-2)} \left(c_0(\tau - \sqrt{\tau^2-1})\right)^{n+m+k-2}\\
&\leq \frac{\vol(\pl B^m)\vol(\pl B^k)}{2\,m\,n^n} (c_0\,c_f)^n (\frac{c_0}{\tau})^{k-2}(\tau^2+1)^{\frac{k-2}{2}}(\frac{c_0}{\tau})^m\vol(\Sm) ~.
\end{align*}
After some simple manipulations, this is exactly the inequality asserted in Theorem \ref{mthm2}.

\section{A remark on the constant} \label{sec_const}

In the Euclidean isoperimetric inequality \eqref{E_isop}, the constant depends only on the dimension of the minimal submanifold, and there is a conjectural sharp constant \eqref{E_const}.  One may be wondering
\begin{enumerate}
\item whether it is possible to remove the dependence of $\tau$ in Theorem \ref{mthm1};
\item intuitively the volume is ``smaller" in the Minkowski space, and maybe the constant \eqref{E_const} could do the job.
\end{enumerate}
The observation below implies that (2) cannot be true, and also suggests that it might be impossible to have the constant solely depends on the dimensions.

For a disk type domain of a minimal surface in $\BR^n$, the sharp isoperimetric inequality, $A\leq \frac{1}{4\pi}L^2$, is known.  Here, $A$ stands for the area of the domain, and $L$ stands for the length of its boundary.

Now, consider a domain of a spacelike, maximal surface in $\BR^{2,1}$.  Note that there are a plethora of such surfaces; see for instance \cite{ref_K} and \cite{ref_ACM}.  By $\BR^{2,1}\hookrightarrow\BR^{2,2}$, Theorem \ref{mthm1} says that
\begin{align*}
A &\leq \frac{(\tau + \sqrt{\tau^2 - 1})^2}{4\pi} L^2 ~.
\end{align*}
Since the normal vector is timelike, the Gauss equation implies that the Gauss curvature of a spacelike, maximal surface in $\BR^{2,1}$ is always non-negative.  However, it is known that the constant for the isoperimetric inequality becomes larger for surfaces with non-negative Gauss curvature.  One can consider a geodesic ball at some point $p$ with radius $\rho<\!\!<1$.  By using the geodesic polar coordinate, one finds that
\begin{align*}
\frac{4\pi A}{L^2} = 1 + \frac{K_p}{4}\rho^2 + o(\rho^3) ~.
\end{align*}
There are also some studies on the isoperimetric inequality for non-negatively curved surfaces.  According to \cite{ref_F,ref_H},
\begin{align*}
A &\leq \frac{1}{4\pi - 2\int\!\!\int K\dd\sm} L^2 ~,
\end{align*}
where the total Gaussian curvature on the domain is a defect.  In \cite{ref_BdC}, this inequality is proved to be sharp.

To sum up, $\frac{1}{4\pi}$ does not work for spacelike, maximal surface in $\BR^{2,1}$.  It also suggests that to remove the dependence on $\tau$, one may have to introduce certain dependence on the intrinsic curvature of the minimal submanifold.


\end{document}